\documentclass[11pt]{amsart}
 \usepackage{amssymb,amsmath,amscd}

\usepackage{pgf,tikz, subfigure, url}
\usepackage{mathrsfs}
\usepackage{epstopdf}
\usepackage{mathtools}
\usetikzlibrary{arrows}
\newtheorem{theorem}{Theorem}[section]
\newtheorem{prop}[theorem]{Proposition}

\newtheorem{coro}[theorem]{Corollary}
\newtheorem*{theorem*}{Claim}
\theoremstyle{definition}
\newtheorem{defi}[theorem]{Definition}
\newtheorem{rem}[theorem]{Remark}
\newtheorem{ex}[theorem]{Example}

\newcommand{\norm}[1]{\left\Vert#1\right\Vert}

\newcommand{\pp}[2]{\frac{\partial#1}{\partial#2}}

\newcommand{\T}{\mathbf{T}}

\begin{document}

\title{Euler flows and singular geometric structures}
\author{Robert Cardona}\address{ Robert Cardona,
Laboratory of Geometry and Dynamical Systems, Department of Mathematics, Universitat Polit\`{e}cnica de Catalunya, Barcelona  \it{e-mail: robert.cardona@upc.edu }
 }
\author{Eva Miranda}\address{ Eva Miranda,
Laboratory of Geometry and Dynamical Systems, Departament of Mathematics, EPSEB, Universitat Polit\`{e}cnica de Catalunya BGSMath Barcelona Graduate School of
Mathematics in Barcelona and
\\ IMCCE, CNRS-UMR8028, Observatoire de Paris, PSL University, Sorbonne
Universit\'{e}, 77 Avenue Denfert-Rochereau,
75014 Paris, France
 }

\author{Daniel Peralta-Salas} \address{Daniel Peralta-Salas, Instituto de Ciencias Matem\'aticas-ICMAT, C/ Nicol\'{a}s Cabrera, nº 13-15 Campus de Cantoblanco, Universidad Aut\'{o}noma de Madrid,
28049 Madrid, Spain \it{e-mail: dperalta@icmat.es} }

\thanks{{ {Robert} Cardona is supported by FPI-BGSMath doctoral grant. {Eva} Miranda  is supported by the Catalan Institution for Research and Advanced Studies via an ICREA Academia Prize 2016 and partially supported  by the grants reference number MTM2015-69135-P (MINECO/FEDER) and reference number {2017SGR932} (AGAUR). Daniel Peralta-Salas is supported by the ERC Starting Grant~335079, the MTM grant 2016-76702-P, and partially supported by the ICMAT--Severo Ochoa grant SEV-2015-0554. This material is based upon work supported by the National Science Foundation under Grant No. DMS-1440140 while Eva Miranda was in residence at the Mathematical Sciences Research Institute in Berkeley, California, during the Fall 2018 semester.}}

\begin{abstract} Tichler proved in \cite{T} that a manifold admitting a smooth non vanishing and closed one-form fibers over a circle. More generally a manifold admitting $k$ independent closed one-forms fibers over a torus $\T^k$. In this article we explain a version of this construction for manifolds with boundary using the techniques of  $b$-calculus \cite{Me,GMP}.
We explore new applications of this idea to Fluid Dynamics and more concretely in the study of stationary solutions of the Euler equations. In the study of Euler flows on manifolds, two dichotomic situations appear. For the first one, in which the Bernoulli function is not constant, we provide a new proof of Arnold's structure theorem and describe $b$-symplectic structures on some of the singular sets of the Bernoulli function. When the Bernoulli function is constant, a correspondence between contact structures with singularities \cite{MO} and what we call $b$-Beltrami fields is established, thus mimicking the classical correspondence between Beltrami fields and contact structures (see for instance \cite{EG}). These results provide a new technique to analyze the geometry of steady fluid flows on non-compact manifolds with cylindrical ends.
\end{abstract}

\maketitle
%
%
%

\section{Introduction}

 The existence of closed one-forms on a manifold simplifies the topology of the manifold in a similar way in which the existence of first integrals of a dynamical system simplifies the topology of its invariant sets. This idea dates back to the work of Tichler who proved in 1970 that a compact manifold admitting a nowhere vanishing closed one-form is a fibration over a circle, or more generally, the existence of $k$ independent closed one-forms implies that the manifold is a fibration over a $k$-dimensional torus. In a dual language, the existence of first integrals also adds constraints on the topology of the invariant manifolds, and the classical Arnold-Liouville theorem shows that an integrable system on a symplectic manifold has tori as compact invariant submanifolds (see \cite{CM} for an application of Tichler's ideas to provide a new proof of Arnold-Liouville theorem).

This same order of ideas can be applied to a more general picture in order to consider Fluid Dynamics and, more concretely, steady Euler flows on manifolds. In particular, we give a new proof of Arnold's structure theorem when the Bernoulli function is not constant, which is based on Tischler's theorem for manifolds with boundary. This starting point takes us to consider manifolds with boundary and $b^{2k}$-forms, thus providing a proof of the $b^{2k}$-Tichler theorem. Additionally, we analyze the singular level sets of the Bernoulli function, which are not considered in Arnold's theorem, and prove that under some assumptions they can be described as $b$-symplectic manifolds. When the Bernoulli function is constant, we reconsider the correspondence between Beltrami fields and contact structures and extend it to contact manifolds with cylindrical ends (compactified as $b$-manifolds) thus obtaining a new correspondence between Beltrami fields in this case with the $b$-contact manifolds recently introduced in \cite{MO}. Several questions concerning the Hamiltonian and Reeb dynamics of $b$-contact manifolds, such as the existence of periodic orbits, can be extremely useful to understand some properties of the stream lines of Beltrami flows on manifolds with cylindrical ends.

\textbf{ Organization of this paper:} In Section 2 we introduce $b^m$-forms and study their desingularization. A special focus is given to the study of $b^m$-symplectic and $b$-contact forms. In Section 3 we prove a Tischler theorem for manifolds with boundary using $b^{2k}$-forms. In the smooth case, the result holds under suitable hypotheses and we use it to provide a new proof, in Section 4, of Arnold's structure theorem for steady Euler flows. We analyze some singular level sets of the Bernoulli function in Section 5, in the context where Arnold's theorem holds. Assuming the Bernoulli function is Morse-Bott, we find singular symplectic structures in some of these sets after resolving their topological singularities. Finally, in Section 6 we study steady Euler flows on manifolds with cylindrical ends and provide a correspondence between Beltrami fields on $b$-manifolds and $b$-contact structures.

{\textbf{Acknowledgements:} We are thankful to the referees of this paper for their careful and efficient work, and their interesting observations.}

\section{A crash course on $b^m$-forms and their desingularization}

In this section we follow closely \cite{GMP} and \cite{MO} to introduce singular symplectic and contact structures that will be of utter relevance in the study of Fluid Dynamics on manifolds with boundary.
\subsection{$b$-symplectic manifolds}
The language of $b$-forms was introduced by Melrose \cite{Me} in order to study manifolds with boundary. The subject gained attention in the realm of Poisson geometry as a special class of Poisson manifolds can be studied using $b$-calculus \cite{GMP}. Most definitions can be used replacing the boundary by any given hypersurface of a manifold without boundary:
\begin{defi}
 A $b$-manifold $(M, Z)$ is an oriented manifold $M$  with an oriented hypersurface
$Z$.
\end{defi}

\begin{rem}
  It is possible to extend this definition to consider non-orientable manifolds. See for instance \cite{GL} and \cite{MP}.
\end{rem}
In order to have the $b$-category we introduce the notion of $b$-map.
\begin{defi}
A $b$-map is a map
$$f : (M_1, Z_1) \longrightarrow (M_2, Z_2)$$
so that $f$ is transverse to $Z_2$ and $f^{-1} (Z_2) = Z_1$.
\end{defi}
Vector fields and differential forms have to be redefined also.
\begin{defi}
A $b$-vector field on a $b$-manifold $(M,Z)$ is a vector field which is tangent to $Z$ at every point $p\in Z$.
\end{defi}

{Observe, in particular, that a $b$-vector field is tangent to the hypersurface $Z$, so from a dynamical point of view $Z$ is an \emph{invariant manifold} by the flow of these vector fields.}
These $b$-vector fields form a Lie subalgebra of vector fields on $M$. Let $t$ be a defining function of $Z$ in a neighborhood $U$ and $(t,x_2,...,x_n)$ be a chart on it. Then the set of $b$-vector fields on $U$ is a free $C^\infty(U)$-module with basis
$$\left( t \pp{}{t}, \pp{}{x_2},\ldots, \pp{}{x_n}\right).$$
We deduce that the sheaf of $b$-vector fields on $M$ is a locally free $C^\infty$-module and therefore it is given by the sections of a vector bundle on $M$. This vector bundle is called the \textbf{$b$-tangent bundle} and denoted by $^b TM$. Its dual bundle is called the \textbf{$b$-cotangent bundle} and is denoted $^bT^*M$.

By considering sections of powers of this bundle we obtain \textbf{$b$-forms}.

\begin{defi}
Let $(M^{2n},Z)$ be a $b$-manifold and $\omega\in \,^b \Omega^2(M)$ a closed $b$-form. We say that $\omega$ is $b$-symplectic if $\omega_p$ is of maximal rank as an element of $\Lambda^2(\,^b T_p^* M)$ for all $p\in M$.
\end{defi}

In the class of Poisson manifolds there is the distinguished subclass of $b$-Poisson manifolds which is indeed formed by $b$-symplectic manifolds together with a bi-vector field naturally associated to the $b$-symplectic forms.

\begin{defi}
Let $(M^{2n},\Pi)$ be an oriented Poisson manifold. Let the map
$$p\in M\mapsto(\Pi(p))^n\in \Lambda^{2n}(TM)$$
be transverse to the zero section. Then $\Pi$ is called a $b$-Poisson structure on $M$. The hypersurface $Z$ where the multivectorfield $\Pi^n$ vanishes,
$$Z=\{p\in M|(\Pi(p))^n=0\}$$
 is called the critical hypersurface of $\Pi$. The pair $(M,\Pi)$ is called a $b$-Poisson manifold.
\end{defi}
The transversality condition is equivalent to saying that $0$ is a regular value of the map $p\longrightarrow (\Pi(p))^n$. The hypersurface $Z$ has a defining function obtained by dividing this map by a non-vanishing section of $\bigwedge^{2n}(TM)$.

The set of $b$-symplectic manifolds is  in one-to-one correspondence with the set of $b$-Poisson manifolds. This correspondence, detailed in \cite{GMP}, can be formulated as
\begin{prop}
A two-form $\omega$ on a $b$-manifold $(M,Z)$ is $b$-symplectic if and only if its dual bivector field $\Pi$ is a $b$-Poisson structure.
\end{prop}

In this context we have a normal form theorem analogous to Darboux theorem for symplectic manifolds. This result is also proved in \cite{GMP}.

\begin{theorem}[{\bf $b$-Darboux theorem}]\label{thm:bdarboux}
Let $(M,Z, \omega)$ be a $b$-symplectic manifold. Then, on a neighborhood of a point $p\in Z$, there exist coordinates $(x_1,y_1,...,x_n,y_n)$  centered at $p$ such that
$$\omega=\frac{1}{x_1}\,dx_1\wedge dy_1 + \sum_{i=2}^{n} dx_i\wedge dy_i.$$
\end{theorem}
Note that with this chart, the symplectic foliation of $(M,\Pi)$ has a specific form. It has two open subsets where the Poisson structure has maximal rank given by $\{x_1>0 \}$ and $\{ x_1<0 \}$. The hyperplane $\{x_1=0\}$ contains leaves of dimension $2n-2$ given by the level sets of $y_1$.

One of the research directions has been to generalize $b$-structures and consider more
degenerate singularities of the Poisson structure. This is the case of $b^m$-Poisson structures, for which $\omega^n$ has a singularity of $A_n$-type in Arnold's list of simple singularities \cite{A2} \cite{A3}. A dual approach is also possible and interesting, working with forms instead of bivector fields.

\begin{defi} A symplectic $b^m$-manifold is a pair $(M^{2n}
, Z)$ with a closed $b^m$-two form $\omega$ which
has maximal rank at every $p\in M$.
\end{defi}
Such as in the $b$-symplectic case,  an analogous $b^m$-Darboux theorem holds. A decomposition for these forms is given in \cite{S}.

\begin{defi} A Laurent Series of a closed $b^m$-form $\omega$ is a decomposition of $\omega$ in a tubular
neighborhood $U$ of $Z$ of the form
	$$ \omega= \frac{dx}{x^m}\wedge ( \sum_{i=0}^{m-1} \pi^*(\hat{\alpha}_i)x^i)+\beta, $$
where $\pi:U\rightarrow Z$ is the projection, where each $\hat{\alpha}_i$ is a closed form on $Z$, and $\beta$ is form on U.
\end{defi}
It is proved in \cite{S} that every closed $b^m$-form admits in a tubular neighborhood $U$ of $Z$ a Laurent form of this type, when fixing a semi-local defining function.

\subsection{$b$-contact manifolds}

Following these ideas and in analogy with contact structures, $b$-contact structures are developed in \cite{MO}.
\begin{defi}
Let $(M,Z)$ be a (2n+1)-dimensional $b$-manifold. A $b$-contact structure is the distribution given by the kernel of a one $b$-form $\xi=\ker \alpha \subset {^b}TM$, $\alpha \in {^b\Omega^1(M)}$, that satisfies $\alpha \wedge (d\alpha)^n \neq 0$ as a section of $\Lambda^{2n+1}(^bT^*M)$. We say that $\alpha$ is a $b$-contact form and the pair $(M,\xi)$ a $b$-contact manifold.
\end{defi}
As in contact geometry one can define the Reeb vector field that satisfies
$$\begin{cases}
i_{R_\alpha}d\alpha=0 \\ \alpha(R_\alpha)=1.
\end{cases}$$
A Darboux type theorem can be proved, providing a normal local form for these structures.
\begin{theorem}\label{bcDarb}
	Let $\alpha$ be a $b$-contact form inducing a $b$-contact structure $\xi$ on a $b$-manifold $(M,Z)$ of dimension $(2n+1)$ and $p\in Z$. We can find a local chart $(\mathcal{U},z,x_1,y_1,\dots,x_n,y_n)$ centered at $p$ such that on $\mathcal{U}$ the hypersurface $Z$ is locally defined by $z=0$ and
	
	\begin{enumerate}
		\item if $R_p \neq 0$
		\begin{enumerate}
			\item  $\xi_p$ is singular, then $$\alpha|_\mathcal{U}=dx_1 +y_1\frac{dz}{z}+ \sum_{i=2}^n x_i dy_i,$$
			\item  $\xi_p$ is regular, then $$\alpha|_\mathcal{U}=dx_1 +y_1\frac{dz}{z}+\frac{dz}{z}+ \sum_{i=2}^n x_i dy_i,$$
		\end{enumerate}
		\item if  $R_p=0$, then $\tilde{\alpha}=f\alpha$ for $f(p)\neq 0$, where $$\tilde{\alpha}_p=\frac{dz}{z} +  \sum_{i=1}^n x_i dy_i.$$
	\end{enumerate}
\end{theorem}

\begin{rem}
There is also a dual correspondence between $b$-contact structures and other structures that play the role of Poisson in the contact context: Jacobi manifolds. The particular subclass is the one of $b$-Jacobi manifolds that satisfy also a transversality condition. For more details you may consult \cite{MO}.
\end{rem}

\subsection{Desingularizing $b^m$-forms}
In \cite{Deblog}  a desingularization procedure for $b^m$-symplectic manifolds associates a family of folded symplectic or symplectic forms to a given $b^m$-symplectic structure depending on the parity of $m$. Namely,

\begin{theorem}[\textbf{Guillemin-Miranda-Weitsman}, \cite{Deblog}]\label{thm:deblogging}
 {Let $\omega$ be a $b^m$-symplectic structure  on a compact orientable manifold $M$  and let $Z$ be its critical hypersurface.
\begin{itemize}
\item If $m=2k$, then there exists  a family of symplectic forms ${\omega_{\epsilon}}$ which coincide with  the $b^{m}$-symplectic form
    $\omega$ outside an $\epsilon$-neighborhood of $Z$ and for which  the family of bivector fields $(\omega_{\epsilon})^{-1}$ converges in
    the $C^{2k-1}$-topology to the Poisson structure $\omega^{-1}$ as $\epsilon\to 0$ .
\item If $m=2k+1$, then  there exists  a family of folded symplectic forms ${\omega_{\epsilon}}$ which coincide with  the $b^{m}$-symplectic form
    $\omega$ outside an $\epsilon$-neighborhood of $Z$.
\end{itemize}}

\end{theorem}

This desingularization can be applied to any $b^m$-form as detailed in \cite{CM2}.

Let us describe how the desingularization works in the even and odd case.

\textbf{Case I: even $m$ .}

Assume $m=2k$ and let $f\in \mathcal{C}^{\infty}(\mathbb R)$ be an odd smooth function such that $f'(x)>0$ for all $x \in [-1,1]$ as shown below,

 and satisfying
\[f(x)=\begin{cases} \frac{-1}{(2k-1) x^{2k-1}}-2 &\textrm{for} \quad x<-1  \\ \frac{-1}{(2k-1) x^{2k-1}}+2 &\textrm{for} \quad x>1
\end{cases}\]
\noindent outside the interval $[-1,1]$.
\begin{figure}
\centering
\includegraphics[scale=.6]{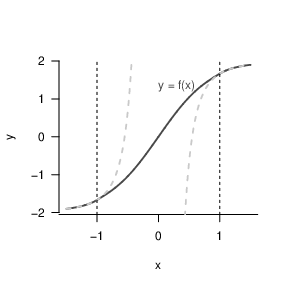}
\label{fig:digraph}
\end{figure}
Scaling the function consider the function
 \begin{equation*}\label{definingequation}f_\epsilon(x):=  \frac{1}{\epsilon^{2k-1}}
f \left(\frac{x}{\epsilon}\right).
\end{equation*}

And outside the interval,
\[f_\epsilon(x)=\begin{cases} \frac{-1}{(2k-1) x^{2k-1}}-\frac{2}{\epsilon^{2k-1}} &\textrm{for} \quad x<-\epsilon  \\ \frac{-1}{
(2k-1)x^{2k-1}}+\frac{2}{\epsilon^{2k-1}}&\textrm{for} \quad x>\epsilon
\end{cases}\]
 Replacing $\frac{dx}{x^{2k}}$ by $df_\epsilon$  in the semi-local expression on $U$ we obtain
 $$  \omega_\epsilon = df_\epsilon \wedge \alpha + \beta. $$
We call this form an $f_\epsilon$-desingularization of $\omega$.

\vspace{3mm}
\textbf{Case II: odd $m$.}
\vspace{3mm}

Consider $m=2k+1$, and consider a function $f\in C^\infty(\mathbb{R})$ satisfying
\begin{itemize}

\item $f(x)=f(-x)$
\item  $f'(x)>0$ if $x>0$
\item $f(x)=x^2-2$ if $x\in [-1,1]$
\item $f(x)=\log(\vert x \vert)$ if $k=0$, $x\in \mathbb R\setminus[-2,2]$
\item $f(x)=-\frac{1}{(2k+2)x^{2k+2}}$ if $k>0$, $x\in \mathbb R\setminus[-2,2]$.
\end{itemize}
\begin{center}
\includegraphics[scale=0.5]{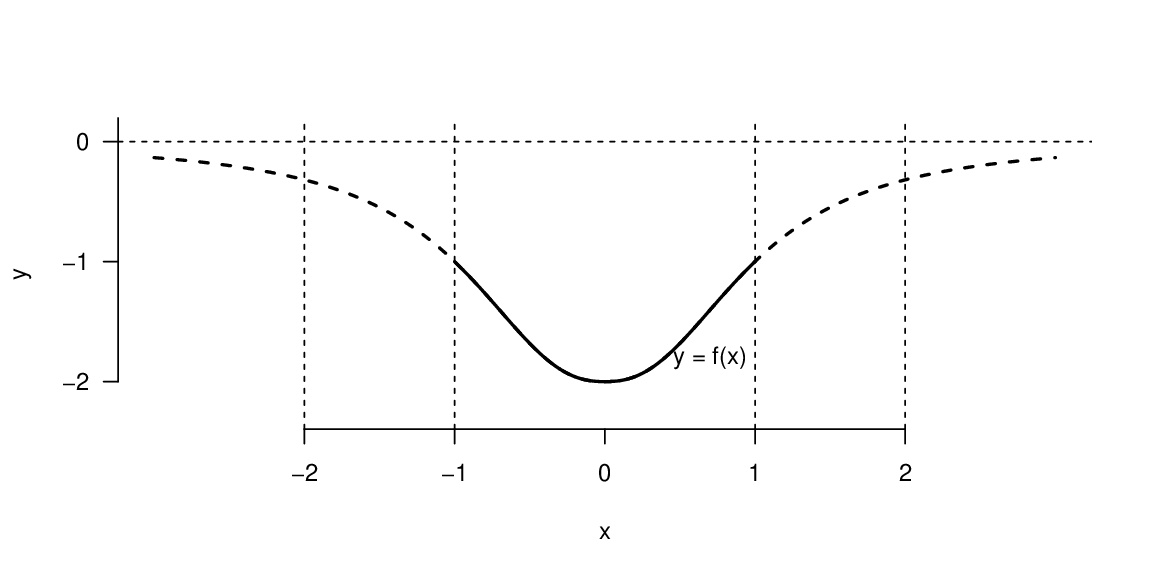}
\end{center}
Taking $\epsilon$  the width of a tubular neighborhood of $Z$ define
\begin{equation*}
f_\epsilon(x):=  \frac{1}{\epsilon^{2k}} f\left(\frac{x}{\epsilon}\right)
\end{equation*}
and consider the form $$ \omega_\epsilon = df_\epsilon \wedge \alpha + \beta. $$
The $f_\epsilon$-desingularization is again smooth and  $df_\epsilon$ vanishes transversally at $Z$.

When $\omega$ is closed, its Laurent decomposition can be used as done in \cite{Deblog} to conclude that $\omega_\epsilon$ is also closed.

\section{A Tischler theorem for manifolds with boundary}

Let us recall Tischler theorem \cite{T} as presented in \cite{CM}.
\begin{theorem}\label{thm:tic}
Let $M^n$ be a closed  manifold endowed with $r$ linearly independent closed $1$-forms $\beta_i, i=1,\dots, r$ which are nowhere vanishing. Then
 $M^n$ fibers over a torus $\T^r$.
\end{theorem}
As a remark in Tischler's original paper, the theorem also holds for compact manifolds with boundary with an extra assumption.
\begin{theorem}\label{thm:ticbound}
Let $M^n$ be a compact connected  manifold with boundary endowed with $r$ linearly independent closed $1$-forms $\beta_i, i=1,\dots, r$ which are nowhere vanishing and satisfy these conditions when restricted to the boundary. Then
 $M^n$ fibers over a torus $\T^r$.
\end{theorem}

Using the language of $b^{2k}$-forms and the deblogging procedure, one can state a Tischler theorem for manifolds with boundary. This theorem gives more information than the one we would get by simply applying the classical Tichler theorem restricted to the boundary.
\begin{defi}
Let $M$ be a manifold with boundary. Its double $\bar M$ is obtained by taking two copies of $M$ and gluing along their boundary.
$$ \bar M = M \times \{0,1\} / \sim , $$
where $(x,0) \sim (x,1)$ for all $x \in \partial M$.
\end{defi}

\begin{theorem}\label{thm:btic}
Let $\alpha_1,...,\alpha_r$ be closed one $b^{2k}$-forms in a $b^{2k}$-manifold $M$ such that $\alpha_1 \wedge \dots \wedge \alpha_r \neq 0$ everywhere in $M$. If the pullback of the  forms to the boundary  are also independent  then $M$ fibers over $\T^r$. Otherwise the double $\bar M$ fibers  over $\T^r$ and the glued boundary fibers over $\T^{r-1}$.
\end{theorem}

\begin{proof}

If the forms  are also independent when pullbacked to the boundary, we can apply the desingularization that we will detail for the second case in the manifold with boundary and apply Theorem \ref{thm:ticbound}.

Otherwise at least one of the forms has a singular part and one considers the extension of the forms $\alpha_i$ into $\bar M$ by symmetry. In this way we obtain a $b^{2k}$-manifold $\bar M$ with critical hypersurface $Z$ where the boundaries have been glued. We can proceed to desingularize  the $1$-forms following \cite{Deblog}. Namely, the forms are closed and  admit Laurent series in a neighborhood $U$ of $Z$,
$$ \alpha_i= (\sum_{j=0}^{2k-1} \alpha^j_i t^j) \frac{dt}{t^{2k}} + \beta_i, $$
for $t$ a positively oriented defining function. Here each $\alpha^j_i$ is a constant function and $\beta_i$ is smooth in $Z$. The term $\alpha_i^0$ is constant and the only non vanishing term of the singular part at the hypersurface $Z$. The rest of terms $\alpha_i^j$ for $j\neq 0$ are paired with powers of $t$ that vanish at $Z$. The dividing term of $\frac{dt}{t^{2k}}$ does not cancel the powers of $t$ because of the structure of the $b^{2k}$-cotangent bundle: one has to think of $\frac{dt}{t^{2k}}$ as if it was a $d\tilde t$ for a coordinate $\tilde t$.

Since at least one of these $\alpha_i^0$ is non vanishing, we can assume $\alpha_1^0 \neq 0$. Redefining
$$\alpha_i:= \alpha_i - \tfrac{\alpha_i^0}{\alpha_1^0}\alpha_1, \text{ for } i=2,...,n$$
we can assume that only the first form has a singular part at the hypersurface and independence of the forms still holds.
Proceeding to the desingularization, one can take a suitable $\epsilon$ and the desingularized forms
$$ \alpha_{i,\epsilon} = df_\epsilon \wedge (\sum_{j=0}^k \alpha^j_i  t^j)+ \beta_i. $$
Since we have $\alpha_1 \wedge ... \wedge \alpha_r \neq 0$,$df_{\epsilon} \neq 0$ and at least one singular form (for instance the first one $\alpha_1^0 \neq 0$) we deduce that  $\alpha_{1,\epsilon}\wedge  ... \wedge \alpha_{r,\epsilon} \neq 0$ using elementary linear algebra as  $\alpha_{i,\epsilon}$  and $\alpha_i$  determine the same matrix of coefficients. One has simply changed the form $\frac{dt}{dt^{2k}}$ of the basis by $df_\epsilon$. Applying Theorem \ref{thm:tic} we deduce that $\bar M$ fibers over $\T^r$. Observe that in $Z$ the form $\alpha_1$ was the only one with a non vanishing singular term. Hence its the only one with a non vanishing term for $df_\epsilon$: we deduce that $\alpha_{2,\epsilon},...,\alpha_{r,\epsilon}$ are independent when restricted to  $Z$ again by linear algebra. In particular, $Z$ fibers over $\T^{r-1}$.

\end{proof}

{\begin{rem} The parity (evenness) of $m$ comes from the desingularization procedure. The desingularized form obtained from a non-vanishing $b^m$-form is non-vanishing only when $m$ is even. For odd $m$, as explained in Section~$2.3$, the resulting form has a zero. This zero cannot be eliminated because the singular part of the form changes sign when crossing the hypersurface. This is why the conditions of the second statement of Theorem~$3.4$ cannot be met for odd $m$. However, the first part can be obtained by adding a constant to the desingularization formula to prevent the desingularized form from vanishing at the boundary.  \end{rem} }

\begin{ex} \normalfont
An easy example to consider is the compact cylinder $C$ visualized as a subset of the torus $\T^2$ (as quotient of the plane $\T^2 \cong (\mathbb{R}/\mathbb{Z})^2$). Consider the $b^{{2k}}$-forms $\frac{1}{\sin (2\pi x)^{2k}} dx$ and $dy$ on $\mathbb{R}^2$. The critical set is the boundary of a compact cylinder. The forms descend to the quotient, and in the compact cylinder they satisfy the hypotheses of the theorem.

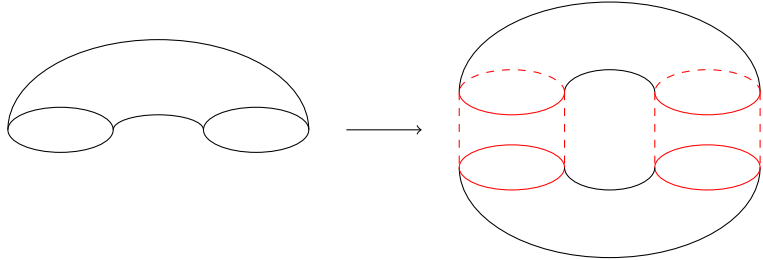
\begin{figure}[!h]

\begin{center}
\begin{tikzpicture}[scale=1]

\draw (0,0,0) arc (0:180:2cm and 1.2cm);
\draw (-0.7,0,0) ellipse (0.7 and 0.3);
\draw (-3.3,0,0) ellipse (0.7 and 0.3);
\draw (-1.4,0,0) arc (0:180:0.6 cm and 0.2cm);

\draw [->] (0.5,0,0) -- (1.5,0,0);

\draw (6,0.5,0) arc (0:180:2cm and 1.2cm);
\draw [color=red](6,0.5,0) arc (0:-180:0.7 and 0.3);
\draw [color=red](3.4,0.5,0) arc (0:-180:0.7 and 0.3);
\draw (4.6,0.5,0) arc (0:180:0.6 cm and 0.3cm);
\draw [dashed,color=red](6,0.5,0) arc (0:180:0.7 and 0.3);
\draw [dashed,color=red](3.4,0.5,0) arc (0:180:0.7 and 0.3);

\draw [dashed, color=red] (6,0.5,0) -- (6,-0.5,0);
\draw [dashed, color=red] (3.4,0.5,0) -- (3.4,-0.5,0);
\draw [dashed, color=red] (4.6,0.5,0) -- (4.6,-0.5,0);
\draw [dashed, color=red] (2,0.5,0) -- (2,-0.5,0);

\draw (6,-0.5,0) arc (0:-180:2cm and 1.2cm);
\draw [color=red](5.3,-0.5,0) ellipse (0.7 and 0.3);
\draw [color=red](2.7,-0.5,0) ellipse (0.7 and 0.3);
\draw (4.6,-0.5,0) arc (0:-180:0.6 cm and 0.3cm);

\end{tikzpicture}
\end{center}
\caption{The double of a compact cylinder}
\end{figure}
\end{ex}

\begin{rem}\label{rem:trans} \normalfont
The second statement can also be applied {for honest De Rham forms} with the following  changes. Instead of one of the forms having a singular part, we ask one of the forms to be transversal to the boundary everywhere. Secondly we need that the forms can be extended to the doubling of the manifold by symmetry which might not be true in general.
\end{rem}

As an easy corollary we obtain,

\begin{coro}
 An $n$-dimensional manifold admitting $n$ independent and closed $b^{2k}$-forms is a compact cylinder $\T^{n-1} \times [0,1]$.
\end{coro}

\section{Euler equations on  3-manifolds}

The Euler equations model the dynamics of an inviscid and incompressible fluid flow on a 3-dimensional manifold, see e.g.~\cite{AKh,P1}. For a smooth domain in $\mathbb{R}^3$ if we denote by $X$ the velocity field of the fluid and $P$ the pressure, which is a scalar function, then the equations can be written as follows,
$$\begin{cases}
	\pp{X}{t} + (X\cdot \nabla)X &= - \nabla P \\
	\operatorname{div}X&=0
\end{cases}.$$
Another vector field that has an important role in fluid dynamics is the vorticity, which is defined as
$$ \omega := \operatorname{curl}X. $$
This vector field is related to the local rotation of the fluid. Using the vorticity, one can rewrite the Euler equations as,
$$\begin{cases}
	\pp{X}{t} - X \times \omega &= - \nabla B \\
	\operatorname{div}X&=0
\end{cases},$$
where $B= P + \frac{1}{2}|X|^2$ is the Bernoulli function.

For any Riemannian $3$-manifold $(M,g)$ one can write the Euler equations
$$\begin{cases}
	\pp{X}{t} + \nabla_X X &= - \nabla P \\
	\operatorname{div}X&=0
\end{cases}.$$
where $\nabla_X$ is the covariant derivative, and the operators $\nabla$ and $\operatorname{div}$ are computed with the metric $g$. Using the Riemannian volume form $\mu$, the second condition can be expressed as follows,
$$ \mathcal{L}_X \mu= 0. $$
The vorticity is then the only vector field satisfying
$$  \iota_\omega \mu= d\alpha, $$
where $\alpha(\cdot)= g(X,\cdot)$ is the dual one-form of $X$ using the Riemannian metric. In terms of the vorticity, the equations read as in the Euclidean case:
$$\begin{cases}
	\pp{X}{t} - X \times \omega &= - \nabla B \\
	\operatorname{div}X&=0
\end{cases},$$
and the Bernoulli function is defined using the Riemannian form, as well as the vector product.

\textbf{Stationary solutions.} We will be interested in equilibrium configurations, i.e. in stationary solutions of these equations.
A well known fact is that the Bernoulli function is a first integral for both $X$ and $\omega$. In particular the stream lines are confined into the level sets of $B$.

The stationary Euler equations with the Bernoulli formulation are
$$ \begin{cases}
X \times \omega = \nabla B \\
 \operatorname{div} X=0
 \end{cases}.$$
In the analytic setting, Arnold noticed the following fact about solutions to these equations. Let $(M,g)$ be an analytic Riemannian manifold and let $X$ be an analytic solution of the equations. If $B$ is constant and $X$ is non-vanishing, the vorticity is proportional to $X$ everywhere, that is $\operatorname{curl}X=fX$ for some analytic function $f$. In this case, $X$ is called a Beltrami flow.

If $B$ is not constant, its critical set $Cr(B):=\{p\in M | \enspace \nabla B(p)=0\}$ has a stratified structure and its codimension is at least 1. In this case, Arnold showed that the structure of the stream lines of $X$ is very similar to the one of integrable systems described by the Arnold-Liouville theorem (see previous sections). We now provide a new proof using the existence of certain closed one-forms as in \cite{CM}.

\begin{theorem}[Arnold's structure theorem]\label{T:Arnold}
Let $X$ be an analytic stationary solution of the Euler equations on an analytic compact manifold with non constant Bernoulli function. The flow is assumed to be tangent to the boundary if there is one. Then there is an analytic set $C$ of codimension at least $1$ such that $M\backslash C$ consists of finitely many domains $M_i$ such that either
\begin{enumerate}
\item $M_i$ is trivially fibered by invariant tori of $X$ and on each torus the flow is conjugated to the linear flow,
\item or $M_i$ is trivially fibered by invariant cylinders of $X$ whose boundaries lie on the boundary of $M$, and all stream lines are periodic.
\end{enumerate}
\end{theorem}


\begin{proof}
We define first the analytic set $C$. Consider $C_1=\{B^{-1}(c):c \text{ is a critical value of }B \}$ and $C_2$ the level sets such that they are tangent at some point to the boundary. Take
$$ C= C_1 \cup C_2. $$
By compactness and analyticity~\cite{AKh}, it is a finite union of level sets of the function $B$ and hence it is an analytic set of codimension greater or equal to one.

Consider  the following one-forms. On the one hand,
$$   \beta= \iota_X \mu_2, $$
where $\mu_2= \iota_{\frac{\nabla B}{|\nabla B|^2}}\mu$
and $\mu$ is the volume in $M$. The form $\mu_2$ is sometimes called the Liouville form and satisfies $\mu= dB\wedge \mu_2$. On the other {hand} consider
$$  \alpha(\cdot)= g(X, \cdot) $$
where $g$ is the Riemannian metric in $M$. We claim that the pullback of these forms to a regular level set $i: N \hookrightarrow M$, $i^*\alpha$ and $i^*\beta$, are closed and independent. We recall that the $2$-form $i^*\mu_2$ is an area-form on $N$.

To prove their independence, first notice that the velocity field $X$ is tangent and non-vanishing on any regular level set of $B$, so
the one-forms $i^*\alpha$ and $i^*\beta$ are non-degenerate on $N$. Since the kernel of $i^*\beta$ is given by $X|_N$, and the kernel of $i^*\alpha$ is transverse to $X|_N$ because $i^*\alpha(X|_N)>0$, we conclude that $i^*\beta$ and $i^*\alpha$ are linearly independent at each point of $N$.


To prove that these one-forms are closed, recall the Euler stationary equations in terms of $\alpha$:
 \begin{equation*}
\begin{cases}
\iota_Xd\alpha= -dB \\
d\iota_X\mu=0
\end{cases}.
\end{equation*}
When restricted to a level set, the first equation implies that $\iota_X d (i^*\alpha)=0$. Since it is a two dimensional submanifold and $X$ is tangent to the level set this yields $d(i^*\alpha)=0$. Observe now that $\mu_2$ satisfies $dB\wedge \mu_2= \mu$. Using that expression and the second Euler equation we obtain,
\begin{align*}
	d\iota_X\mu &= d( \iota_X(dB\wedge \mu_2))= d( \iota_X dB\wedge \mu_2 - dB \wedge \iota_X\mu_2) \\
	&= -d(dB \wedge \iota_X\mu_2)= - dB \wedge d\iota_X\mu_2=0.
\end{align*}
This equality stands everywhere. Since in the neighborhood of the regular level set we have that $dB\neq 0$, we infer that $d\iota_X\mu_2= dB\wedge \gamma$ for some one-form $\gamma$. Accordingly, we obtain that $d(i^*\iota_X\mu_2)=0$, thus proving that $i^*\beta$ is also closed.

Now, suppose that $N$ has no boundary component. Then applying Theorem \ref{thm:tic} we deduce that it is a torus. If $N$ has a boundary, it must lie on $\partial M$ and since it is invariant under a non-vanishing field $X$, the boundary consists of finitely many periodic orbits. The fact that $X$ is non-vanishing and tangent to the boundary of the level set, implies that the pullback of $\alpha$ to the boundary is non-vanishing as well. By Remark \ref{rem:trans}  the first case of Theorem \ref{thm:btic} can be applied. Hence $N$ is an orientable surface with boundary that fibers over $S^1$, thus it is a  cylinder. This determines the topology of the regular level sets of $B$. The rest of the proof is standard. Indeed, let $\phi_t$ be the flow of the vector field $S= \frac{\nabla B}{dB(\nabla B)}$, which satisfies $dB(S)=1$. Then we have
$$ \pp{}{t}B(\phi_t(x))=dB(S)=1, \enspace B(\phi_0(x))=B(x). $$
We deduce that $B(\phi_t(x))= B(x)+t$ and hence the open set $M_i$ is a trivial fibration $\T^2 \times I$ for a real interval $I$ in the case that the level sets have no boundary. The same holds when the level sets are cylinders (due to the fact that in the complement of the set $C$ the level sets of $B$ intersecting the boundary have a transverse intersection). Since the vector field $X$ commutes with $\text{curl} X$ it follows that it is conjugated to a linear flow on each level set diffeomorphic to a torus. For the cylinder, all orbits are periodic as an easy consequence of the Poincar\'e-Bendixon theorem and the fact that $X$ preserves the area form $i^*\mu_2$.
\end{proof}
\begin{rem} This proof also works to obtain the topology of the regular level sets in the four dimensional Euler equations studied in \cite{GK}. The way to obtain the closed and independent one forms is done as in \cite{CM}.
\end{rem}

Arnold's theorem shows that for non constant Bernoulli functions, the situation is very similar to integrable systems. However for a constant $B$ a contact structure appears and the situation is the opposite: a non integrable one. This case will be analyzed in Section~\ref{S:Reeb}.

\section{Geometric structures on singular level sets}

In this section we would like to understand the geometric structure induced on some of the singular level sets of the Bernoulli function. In the analytic setting a lot can be said about the structure of the level sets of $B$, both regular or singular with some assumptions, as studied in \cite{A1} and \cite{CV}. In  Proposition 2.6 in \cite{CV}  a topological classification of the singular level sets is obtained when $B$ is analytic and $X$ is assumed to be non-vanishing. Namely,

\begin{prop}
Let $B$ be analytic and $X$ a non-vanishing Euler flow on a closed $3$-manifold $M$. Then, each singular level set of $B$ (finitely many) is a finite union of embedded $X$-invariant sets that are periodic orbits, $2$-tori, Klein bottles, open cylinders or open M\"obius strips.
\end{prop}

In this section we shall assume that $B$ is a Morse-Bott function instead of analytic and we shall not impose any assumptions on $X$. {The standard Arnold's theorem is studied under analyticity assumptions. The next natural scenario would be to consider Morse-Bott functions as they are well-behaved at the critical set and are dense in the set of smooth functions. This assumption is not uncommon in our context: for instance Arnold's structure theorem is known to hold for Morse-Bott Bernoulli functions if the manifold has no boundary. In~\cite{EG2} the assumption considered is that of stratified singularities, which would go one step further.} For Morse-Bott singular level sets we will see that some $b$-symplectic structures appear, which provides a new connection between these Poisson structures and physics. For these structures to appear, we need the existence of a singular submanifold in a level set thay might end up being the critical hypersurface of a $b$-symplectic manifold. For this to make sense we need level sets with a regular part and a singular one. Since the singularities are of Morse-Bott type, the only two options that admit this structure are the following two local forms for $B$:
\begin{enumerate}
\item An isolated critical point of saddle type: $B= x^2 + y^2 - z^2$.
\item A 1-dimensional critical set of saddle type: $B= x^2 - y^2$.
\end{enumerate}

We are interested in the case where the level set is compact, which will have a topological singularity. In the first case the singularity is a point in a surface. In the second case the singularity is a circle. In both cases there is a topological desingularization to obtain a manifold with a codimension one singular submanifold. The structure that we are interested in is the following. If $\mu$ is the Riemannian volume in the $3$-manifold then the area form preserved by $X$ in a level set of $B$ is $i^*\mu_2$, as explained in the proof of Theorem~\ref{T:Arnold}, where
$$  \mu_2 = \iota_{\frac{\nabla B}{|\nabla B|^2}}\mu, $$
and $i$ is the inclusion of the level set of $B$ into $M$. Let us answer to the following question: what kind of geometric structure is $i^*\mu_2$ in these desingularized singular level sets?

\paragraph{Case 1} Consider that $B$ around the singularity looks like $B= x^2+y^2-z^2$. The volume form will be locally $\mu= dx\wedge dy \wedge dz$. The gradient of $B$ is $(2x,2y,-2z)$ and hence the vector field we are interested in is
$$X= \frac{\nabla B}{|\nabla B|^2}= (\frac{x}{x^2+y^2+z^2}) \pp{}{x} + (\frac{y}{x^2+y^2+z^2}) \pp{}{y} - (\frac{z}{x^2+y^2+z^2}) \pp{}{z}.  $$

Denoting $r^2= x^2+y^2+z^2$ and computing the two-form we obtain
$$ \mu_2 = \frac{x}{r^2} dy\wedge dz - \frac{y}{r^2} dx \wedge dz - \frac{z}{r^2}dx\wedge dy. $$

Let $i:N \hookrightarrow M$ be the inclusion of the level set $N$ into $M$, in coordinates $i:(\theta, \omega) \mapsto ( \omega \cos{\theta}, \omega \sin{\theta},\omega)$. A simple computation yields,
$$ i^*\mu_2=  d\theta \wedge d\omega. $$
This already extends to an area form, one can think of it as a polar blow-up. However, to end up with a concrete smooth manifold we can also realize the topological singularity in a cylinder.
Let $\sigma$ be the desingularization
\begin{equation*}
\sigma: \begin{cases}
x= u.w \\
y=v.w \\
z=w
\end{cases}.
\end{equation*}

This desingularization sends the cone to a cylinder.

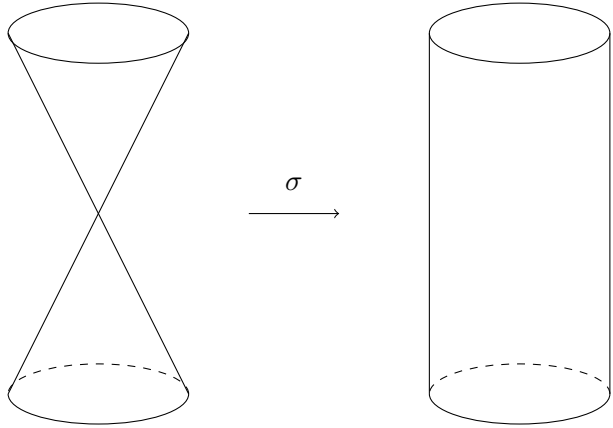
\begin{figure}[!h]
\begin{center}
\begin{tikzpicture}[scale=0.8]

\draw (0,3,0) ellipse (1.5cm and 0.5cm);
\draw (-1.5,3,0)--(1.5,-3,0);
\draw (1.5,3,0) -- (-1.5,-3,0);
\draw [dashed](1.5,-3,0) arc (0:180:1.5cm and 0.5cm);
\draw (1.5,-3,0) arc (0:-180:1.5cm and 0.5cm);

\draw [->] (2.5,0,0)--(4,0,0);
\draw (3.25,0.5,0) node{$\sigma$} ;

\draw (7,3,0) ellipse (1.5cm and 0.5cm);
\draw (5.5,3,0)--(5.5,-3,0);
\draw (8.5,3,0) -- (8.5,-3,0);
\draw [dashed](8.5,-3,0) arc (0:180:1.5cm and 0.5cm);
\draw (8.5,-3,0) arc (0:-180:1.5cm and 0.5cm);

\end{tikzpicture}
\end{center}
\caption{Case $1$ desingularization}
\end{figure}
Letting $j: (\theta,\omega) \mapsto (\cos \theta, \sin \theta, \omega)$ be the inclusion of the cylinder, we obtain
$$ j^*\sigma^*\mu_2=   d\theta \wedge d\omega, $$
which is a symplectic structure. This is a local model but using bump functions one obtains a symplectic surface globally defined.

\paragraph{Case 2} Consider now a point in a 1-dimensional critical set of $B$ of saddle type; hence the function looks locally as $B= x^2-y^2$. Again the volume form will be written $\mu= dx\wedge dy \wedge dz$. The gradient of $B$ is $(2x,-2y,0)$ and the vector field is
$$ X= ( \frac{x}{x^2+y^2} ) \pp{}{x} - (\frac{y}{x^2+y^2} )\pp{}{y}. $$
Denoting $r^2= x^2+ y^2$ the two form is $\mu_2= \frac{x}{r^2}dy\wedge dz + \frac{y}{r^2}dx\wedge dz$. The desingularization applied now is
\begin{equation*}
\sigma: \begin{cases}
x= u.v \\
y=v\\
z=w
\end{cases}.
\end{equation*}
It sends the two intersecting planes to two separate ones. We are realizing the topological singularity of the level, and hence forgetting now about the function $B$. We analyze the structure of $\mu_2$ after desingularization and restricted to the level set.

\begin{figure}[!h]
\begin{center}
\begin{tikzpicture}[scale=0.8]

\draw (-2,2,0) -- (2,1,0);
\draw (-2.4,1,0) -- (2.4,2,0);
\draw(-2.4,1,0) -- (-2.4,-4,0);
\draw(2.4,2,0) -- (2.4,-3,0);

\draw [dashed] (-2,1.1,0) -- (-2,-3,0);
\draw (2,1,0) -- (2,-4,0);

\draw(-2.4,-4,0) -- (0,-3.5,0);
\draw[dashed](0,-3.5,0) -- (2,-3.1,0);

\draw (2,-4,0)--(0,-3.5,0);
\draw [dashed](0,-3.5,0)--(-2,-3,0);

\draw (-2,2,0)--(-2,1.1,0);
\draw (2.4,-3,0)--(2,-3.1);

\draw [red](0,1.5,0)--(0,-3.5,0);

\draw[->] (3.3,-1.2,0)--(4.7,-1.2,0);
\draw (4,-0.9,0) node{$\sigma$};


\draw (6,2,0) -- (6,1,0);
\draw [dashed](6,1,0) -- (6,-3,0);
\draw (10,1,0) -- (10,-4,0);

\draw(5.6,1,0) -- (5.6,-4,0);
\draw(10.4,2,0) -- (10.4,-3,0);

\draw(5.6,1,0)--(10,1,0);
\draw(5.6,-4,0)--(10,-4,0);

\draw(6,2,0)--(10.4,2,0);
\draw[dashed](6,-3,0)--(10,-3,0);

\draw (10,-3,0)--(10.4,-3,0);

\draw [red] (7.8,1,0)--(7.8,-4,0);
\draw [red,dashed] (8.2,2,0)--(8.2,-3,0);

\end{tikzpicture}
\end{center}
\caption{Case $2$ desingularization}
\end{figure}
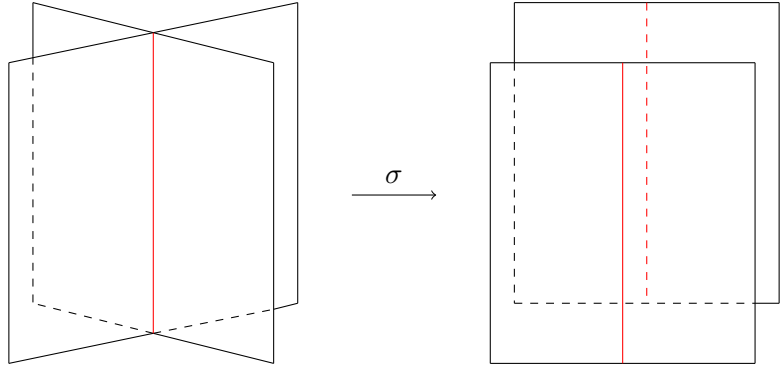
If $j$ is the inclusion of any of the two planes, then we have
$$ j^*\sigma^*\mu_2= \frac{1}{v} dv \wedge d\omega, $$
where the change of sign depends on which one of the planes we consider. One could be confused by the change of induced orientations on the hyperplanes outside of the critical curve: However $\mu_2$ was already well defined in the regular part of the level and there is not a problem of sign.

This model is local, but since the function is Morse-Bott, the desingularization is applied through a circle. Assuming that the negative (and positive) normal bundles of the singular set are orientable (to avoid problems as in \cite{Sm}), the normal form $B=x^2-y^2$ holds on a neighborhood of the critical circle (for appropriate coordinates). Therefore one obtains globally a $b$-surface. Accordingly, we have produced a $b$-symplectic structure on each component, and globally a $b$-surface having two circles as critical set.

For the sake of simplicity in the analysis, we have used a model where the metric looks like the Euclidean one near the singular sets. This is true for \emph{nice metrics} with respect to the Morse-Bott function $B$ as the ones introduced in Hutching's thesis \cite{H1, H2}. Nevertheless, the qualitative picture described above is independent of this choice.

For our purposes, the most interesting situation is Case~2, where the singular locus is a whole curve. Observe that the area form $\mu_2$ always satisfies the following identity:
\begin{equation} \label{equ}
dB\wedge \mu_2= \mu.
\end{equation}
By construction, when restricted to the planes obtained after the desingularization procedure, the form $i^*\mu_2$ is an area form that goes to infinity when approaching the critical curve $Z$. Letting $i:N \hookrightarrow M$ be the inclusion of any of the two planes $\{x=y\}$ or $\{x=-y\}$ we have
$$i^*\mu_2 = \frac{1}{f} \omega,$$
for a function $f\in N$ that vanishes along $Z$ and an area form $\omega$. Taking coordinates such that $\omega= du \wedge dv$ we can consider the dual vector field $\Pi_2= f \pp{}{u}\wedge \pp{}{v}$. Observe also that equation (\ref{equ}) holds everywhere and $dB= 2x dx - 2ydy$ induces different orientations on each side of $Z$ inside $N$.
\begin{figure}[!h]
\begin{center}
\begin{tikzpicture}[scale=0.8]

\draw (-2,2,0) -- (2,1,0);
\draw (-2.4,1,0) -- (2.4,2,0);
\draw(-2.4,1,0) -- (-2.4,-4,0);
\draw(2.4,2,0) -- (2.4,-3,0);

\draw [dashed] (-2,1.1,0) -- (-2,-3,0);
\draw (2,1,0) -- (2,-4,0);

\draw(-2.4,-4,0) -- (0,-3.5,0);
\draw[dashed](0,-3.5,0) -- (2,-3.1,0);

\draw (2,-4,0)--(0,-3.5,0);
\draw [dashed](0,-3.5,0)--(-2,-3,0);

\draw (-2,2,0)--(-2,1.1,0);
\draw (2.4,-3,0)--(2,-3.1);
\draw [red](0,1.5,0)--(0,-3.5,0);

\draw [->,opacity=0.8] (0,-3.5,0)--(0,-3.5,-2);
\draw (0,-3.5,-2.5) node{$y$};
\draw [->,opacity=0.8] (0,-3.5,0)--(1.5,-3.5,0);
\draw (1.7,-3.6,0) node{$x$};

\draw [|->] (1.7,1.45,0)--(2.3,1.35,0);

\draw [|->] (-1.3,1.5,0)--(-1.9,1.4,0);

\draw [|->] (-1.1,-1.4,0)--(-1.6,-1.1,0);

\draw [|->] (0.8,-0.9,0)--(1.3,-0.6,0);
%

\end{tikzpicture}
\end{center}
\caption{Orientation induced by $dB$}
\end{figure}
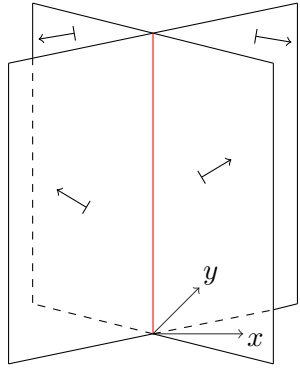

Also $dB$ vanishes when restricted to one of the planes $\{x=y\}$ or $\{x=-y\}$ in first order. This ensures that the pole of $\mu_2$ is of order one since $\mu$ in $M$ is a volume form. In the desingularized manifold, which is a surface, the bivector field defines a Poisson structure that vanishes along a curve in order $1$ (which we call critical curve). Thus we have obtained a $b$-symplectic structure on the desingularized surface. Thus proving,
\begin{prop}
Singular sets of the second kind can be desingularized into surfaces with a $b$-symplectic structure that is preserved by the flow of the fluid.
\end{prop}

\begin{rem}
By using regularization-type techniques like in \cite{MO}, one can produce artificially singularities or order $2k+1$ for any $k\in \mathbb{N}$.
\end{rem}

\section{Beltrami fields in $b$-manifolds}\label{S:Reeb}

When the Bernoulli function is constant,  Beltrami fields are obtained. These are vector fields that are parallel to their vorticity i.e. $\operatorname{curl}X=fX$ for a function $f\in C^{\infty}(M)$. When $f$ and $X$ are non-vanishing we speak about nonsingular rotational Beltrami fields. In \cite{EG} a correspondence between these fields and rescaled Reeb vector fields of contact structures is established.

We recall that the motivation for $b$-manifolds is studying manifolds with cylindrical ends. The critical surface captures the asymptotic behavior of geometric structures. We see a Riemannian manifold with a cylindrical end as the interior of a compact manifold with boundary.

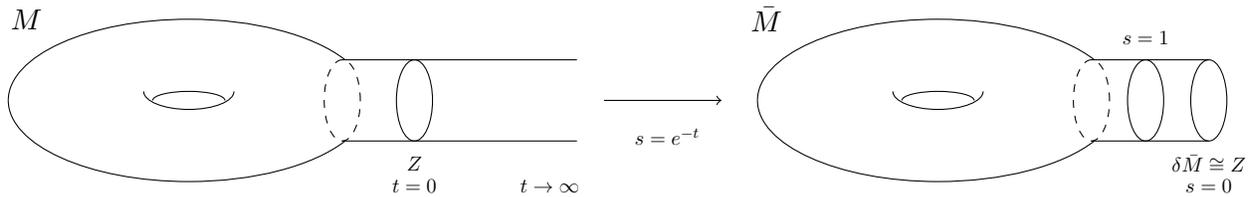
\begin{figure}[!h]
\begin{center}
\begin{tikzpicture}[scale=1.2]

\draw (-1.8,0.9,0) node[scale=1]{$M$};
\draw (6.4,0.9,0) node[scale=1]{$\bar M$};

\draw (-2,0,0) arc (-180:-30:2 and 0.9);
\draw (-2,0,0) arc (-180:-330:2 and 0.9);

\draw (0.4,0,0) arc (0:180:0.4 and 0.1);
\draw (-0.5,0.1,0) arc (-180:0:0.5 and 0.2);

\draw (1.7,0,0)[dashed] ellipse (0.2 and 0.45);

\draw (1.7,0.45,0)--(4.3,0.45,0);
\draw (1.7,-0.45,0)--(4.3,-0.45,0);

\draw (2.5,0,0) ellipse (0.2 and 0.45);
\draw (2.5,-0.7,0) node[scale=0.7]{$Z$};
\draw (2.5,-0.95,0) node[scale=0.7]{$t=0$};
\draw (4,-0.95,0) node[scale=0.7]{$t\rightarrow \infty$};


\draw [->] (4.6,0,0) -- (5.9,0,0);
\draw (5.3, -0.4,0) node[scale=0.7]{$ s=e^{-t} $};

\draw (6.3,0,0) arc (-180:-30:2 and 0.9);
\draw (6.3,0,0) arc (-180:-330:2 and 0.9);

\draw (8.7,0,0) arc (0:180:0.4 and 0.1);
\draw (7.8,0.1,0) arc (-180:0:0.5 and 0.2);

\draw (10,0,0)[dashed] ellipse (0.2 and 0.45);

\draw (10,0.45,0)--(11.3,0.45,0);
\draw (10,-0.45,0)--(11.3,-0.45,0);

\draw (10.6,0,0) ellipse (0.2 and 0.45);
\draw (10.6,0.7,0) node[scale=0.7]{$s=1$};

\draw (11.3,0,0) ellipse (0.2 and 0.45);
\draw (11.3,-0.95,0) node[scale=0.7]{$s=0$};
\draw (11.3,-0.7,0) node[scale=0.7]{$\delta \bar M \cong Z$};

\end{tikzpicture}
\end{center}
 \caption{Compactification to a $b$-manifold}
 \end{figure}

 If we take the Euler equations in a manifold of this kind, one can consider them after the transformation to a $b$-manifold. The equations obtained are the same but with a resulting $b$-metric $g$ and $b$-volume form $\mu$ that capture the asymptotical behavior of the geometric structures. Now working in the $b$-tangent and cotangent bundles, in terms of the form $\alpha(\cdot)= g(X, \cdot)$ and the Bernoulli function, the Euler equations are still of the form:
 \begin{equation*}
\begin{cases}
\iota_Xd\alpha= -dB \\
d\iota_X\mu=0,
\end{cases}
\end{equation*}
The special case of Beltrami fields is,
\begin{equation*}
\begin{cases}
d\alpha= f \iota_X\mu \\
d\iota_X\mu=0,
\end{cases}
\end{equation*}
for a non-vanishing function $f\in C^\infty (M)$. Following similar arguments as in the contact case, we can prove a correspondence between Beltrami fields in $b$-manifolds and $b$-contact structures. This is the formalization of the following idea: in the cylindrical manifold we can apply the usual correspondence of Beltrami fields with contact structures. When obtaining the $b$-manifold the interior admits again a contact structure, but the $b$-manifold looks only locally as the original manifold. Globally speaking the $b$-contact structure gives more information about the global asymptotic behavior close to the boundary.
\begin{rem}
For the discussion in this section let us put emphasis in a particularity of $b$-vector fields illustrating it with an example. Consider the $b$-manifold $(\mathbb{R}^2, Z=\{ (x,y) | x=0 \})$ with basis $\langle \pp{}{y}, x\pp{}{x} \rangle$. The $b$-vector field $X= x\pp{}{x}$ vanishes in the usual sense of a vector field when $x=0$. However as a section of the $b$-tangent bundle the term $(x\pp{}{x})$ is not vanishing. When paired with the dual form $\alpha= \frac{dx}{x}$ it satisfies $\alpha(X)=1$ even in $Z$.
\end{rem}
The statement of the $b$-Beltrami fields and $b$-contact correspondence is now presented.
\begin{theorem}
Let $M$ be a $b$-manifold of dimension three. Any rotational Beltrami field and non-vanishing as a section of $\prescript{b}{}{TM}$ on $M$ is a Reeb vector field (up to rescaling) for some $b$-contact form on $M$. Conversely given a $b$-contact form $\alpha$ with Reeb vector field $X$ then any nonzero rescaling of $X$ is a rotational Beltrami field for some $b$-metric and $b$-volume form on $M$.
\end{theorem}

\begin{proof}
The proof is very similar to the one for usual Beltrami fields. One just needs to work with the $b$-tangent bundle $^b TM$ and its dual instead of the tangent bundle. Let $X$ be a Beltrami field in $(M,Z)$, a $b$-manifold of dimension three. For this implication we can follow \cite{EG}. Denote $e_1=X/\norm{X}$ which is globally defined as a $b$-vector field since $X$ is a non-vanishing section of $^bTM$ and take a couple $e_2,e_3$ to have an orthonormal frame. Then consider $\alpha(\cdot)= g(X,\cdot)= \norm{X}e^1$ where $e^1$ is the dual to $e_1$. This form is the dual of a $b$-vector field by a $b$-metric and hence defines a one $b$-form. Recall that $d\alpha= f\iota_X\mu$ which is also a $b$-form of degree $2$ since it is a contraction of a $b$-vector field by a $b$-volume form.

The $b$-volume form has the form $\mu= h e^1\wedge e^2 \wedge e^3$ for a non-vanishing function $h\in C^\infty(M)$. Then it is clear that
$$ \alpha \wedge d\alpha= g(X,\cdot)\wedge f\iota_x\mu=fh\norm{X}^2e^1\wedge e^2\wedge e^3 \neq 0. $$
Also $\iota_Xd\alpha= f\iota_X\iota_X\mu=0$ and $X$ is a rescaled Reeb vector field of a $b$-contact structure given by $\alpha$.

Conversely, consider a $b$-contact form $\alpha$ and a rescaling of its Reeb vector field $Y=hR$ for $h\in C^\infty(M)$ a non-vanishing function. We will follow the idea in \cite{G2} for this implication. Using the Darboux theorem for $b$-contact forms, Theorem \ref{bcDarb}, the subbundle $\ker \alpha$ is generated by $\xi_p= \langle z\pp{}{z}, \pp{}{y_1} \rangle$ or $\xi_p= \langle -\pp{}{x_1} + z\pp{}{z}, \pp{}{y_1} \rangle $ if the Reeb vector field does not vanish as a smooth vector field. When $R$ vanishes as a smooth vector field (but not as a $^b TM$ section) then $R=z\pp{}{z}$ and $\xi_p= \langle \pp{}{x_1}, x_1 z\pp{}{z} - \pp{}{y_1} \rangle$. Hence $\xi$ is a vector bundle of constant rank $2$ over $M$, it is indeed a subbundle of $^bTM$. Recall now from \cite{McS} the following proposition:
\begin{prop}
Let $E \rightarrow M$ be a $2n$-dimensional vector bundle with a non degenerate bilinear form $\omega_q$ in each fiber $E_q$ which varies smoothly with $q\in M$. Then there exists an almost complex structure which is compatible with $\omega$, i.e. such that $\omega(\cdot, J \cdot)$ is positive definite.
\end{prop}
Let $g$ be the $b$-metric
$$ g(u,v)= \frac{1}{h} (\alpha(u) \otimes \alpha(v))+ d\alpha(u,Jv). $$
The vector field Y satisfies $\iota_Yg=\alpha$. Take $\mu= \frac{1}{h}\alpha\wedge d\alpha$ as $b$-volume form in $M$. It obviously satisfies $\iota_Y\mu=d\alpha$. Hence $Y$ is a Beltrami field (with constant proportionality factor) for this choice of $g$ and $\mu$.

%

\end{proof}

\begin{ex} [ABC flows]
A very well-known family of Beltrami flows in $\T^3$ are the ABC flows:
$$ X(x,y,z)= [A\sin z + C \cos y] \pp{}{x} + [B\sin x + A\cos z] \pp{}{y} + [C\sin y + B\cos x]\pp{}{z}.  $$
Everything is computed in $\mathbb{R}^3$ and then quotiented depending on which hypersurface we consider. Taking as hypersurface $Z=\{z=0 \}$ one can check for which values of the parameters the $b$-vector field
$$ X(x,y,z)= [A\sin z + C \cos y] \pp{}{x} + [B\sin x + A\cos z] \pp{}{y} + [C\sin y + B\cos x]z\pp{}{z} $$
is a Beltrami field in the corresponding $b$-manifold. The metric and volume forms are
$$ g= dx^2 + dy^2 + (\frac{dz}{z})^2, \enspace \mu= dx\wedge dy \wedge \frac{dz}{z}. $$
We compute the one form
$$\alpha= g(X,\cdot) = [A\sin z + C \cos y] dx + [B\sin x + A\cos z] dy + [C\sin y + B\cos x]\frac{dz}{z},$$
and the contraction by the volume
$$ \iota_X\mu=[-B\sin x - A\cos z] dx\wedge \frac{dz}{z}+ [A\sin z + C \cos y] dy\wedge \frac{dz}{z} + [C\sin y + B\cos x] dx\wedge dy. $$
It is clear that $d\iota_X\mu=0$, it remains to check the equation $d\alpha= f\iota_X \mu$. Computing the derivative of alpha
$$ d\alpha= [ - B \sin x -zA\cos z] dx\wedge \frac{dz}{z} + [zA\sin z + C\cos y   ] dy\wedge \frac{dz}{z}+[ C\sin y + B\cos x ] dx\wedge dy. $$
When differentiating with respect to $z$ in the $b$-cotangent bundle, a $z$ factor appears. For $d\alpha= f \iota_X \mu$ to be satisfied, we need $A=0$ and $f=1$. The two-parameter family of vector fields
$$ X(x,y,z)= C\cos y \pp{}{x} + B\sin x \pp{}{y} + [C\sin y + B \cos x]z\pp{}{z} $$
is $b$-Beltrami on the $b$-manifold $\T^2\times \mathbb{R}$ with a $\T^2$ as critical hypersurface. To obtain a vector field in a compact manifold, one can chose $\sin z$ instead of $z$ as defining function of the critical set (which is now defined in the quotient to $\T^3$) and hence work with $\sin z \pp{}{z}$ and $\frac{dz}{\sin z}$. We obtain a Beltrami field on $\T^3$ with two $\T^2$ as critical hypersurfaces. It is an easy computation to check that the $b$-vector field $X$ is non-vanishing as a section of $^b TM$ if and only if $|B| \ne |C|$.

\begin{rem}
For this $b$-manifold the corresponding original manifold can be thought as $M=\mathbb{R} \times \T^2$. When compactifying each of the cylindrical ends we obtain a manifold diffeomorphic to $\T^2 \times [0,1]$. When considering its double the resulting $b$-manifold is $\T^3$ with two $\T^2$ as critical hypersurfaces.
\end{rem}

As an example of $b$-contact structure, let us compute it in the simple case $C=0$ and $B > 0$ for ABC fields in $\T^3$, i.e. the defining function is $\sin z$. The one $b$-form $\alpha$ in this case is
$$ \alpha= g(X,\cdot)= B\sin x dy + B\cos x \frac{dz}{\sin z}. $$
It is clearly a $b$-contact structure since
$$ \alpha\wedge d\alpha= B^2 dx\wedge dy\wedge \frac{dz}{\sin z}, $$
and its Reeb vector field is $R= \frac{1}{B}\sin x \pp{}{y} + \frac{1}{B} \cos x \sin z \pp{}{z}$, a rescaling of the original Beltrami field.
\end{ex}

\begin{rem}
This correspondence holds true if we consider Beltrami fields on $b^m$-manifolds. The associated structure is then a $b^m$-contact structure.
\end{rem}

The Weinstein conjecture \cite{weinstein} on periodic orbits of Reeb flows claims that any Reeb vector field admits a periodic orbit on a compact manifold.
In \cite{MO} a plug-like construction is used to give a counterexample to the Weinstein conjecture on a $b$-contact manifold: the associated Reeb vector field does not have a smooth periodic orbit and the notion of singular periodic orbits is introduced.
\begin{coro} A Beltrami field on a $b$-manifold does not  necessarily have a smooth periodic orbit.
\end{coro}
In \cite{MO} the authors conjecture a \emph{singular version of the Weinstein conjecture} claiming that the Reeb vector field of any compact $b$-contact manifold  possesses at least one periodic orbit which may be singular in the following sense:
\begin{defi}
Let $M$ be a manifold with hypersurface $Z$. A singular periodic orbit is either a periodic orbit in $M \backslash Z$ or an orbit $\gamma$ such that $\operatorname{lim}_{t\rightarrow \pm \infty} \gamma(t) \in Z$.
\end{defi}
One could obtain information about the stream lines of a $b$-Beltrami flow depending on the possible casuistics that this conjecture opens. In particular, this would allow to establish the existence of either a periodic orbit or an unbounded orbit that escapes (in both directions) through a cylindrical end. 
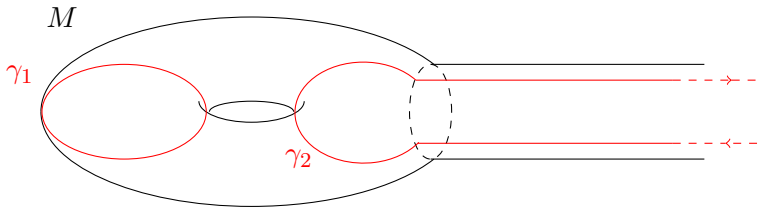
\begin{figure}[!h]
\begin{center}
\begin{tikzpicture}[scale=1.4]

\draw (-1.8,0.9,0) node[scale=1]{$M$};

\draw (-2,0,0) arc (-180:-30:2 and 0.9);
\draw (-2,0,0) arc (-180:-330:2 and 0.9);

\draw (0.4,0,0) arc (0:180:0.4 and 0.1);
\draw (-0.5,0.1,0) arc (-180:0:0.5 and 0.2);

\draw (1.7,0,0)[dashed] ellipse (0.2 and 0.45);

\draw (1.7,0.45,0)--(4.3,0.45,0);
\draw (1.7,-0.45,0)--(4.3,-0.45,0);

\draw [red] (-1.21,0,0) ellipse (0.78 and 0.45);

\draw [red] (1.56,0.3,0)--(4,0.3,0);
\draw [red] (1.58,-0.3,0)--(4,-0.3,0);
\draw [red] (1.56,0.3,0) arc (40:324:0.65 and 0.48);
\draw [red,dashed, ->] (4,0.3,0)--(4.56,0.3,0);
\draw [red,dashed, ->] (4.8,-0.3,0)--(4.5,-0.3,0);
\draw [red,dashed] (4.56,0.3,0)--(4.8,0.3,0);
\draw [red,dashed] (4,-0.3,0)--(4.5,-0.3,0);

\draw (0.45,-0.45,0)[red] node[scale=1]{$\gamma_2$};
\draw (-2.2,0.35,0)[red] node[scale=1]{$\gamma_1$};
\end{tikzpicture}
\end{center}
 \caption{Two possible singular periodic orbits}
 \end{figure}


\begin{thebibliography}{99}


\bibitem{A1} V. I. Arnold, \emph{Sur la topologie des \'ecoulements stationnaires des fluides parfaits}.
C. R. Acad. Sci. Paris 261 (1965) 17-20.

\bibitem{A2} V. I. Arnold, \textit{Remarks on Poisson structures on a plane and on other powers of volume elements}. (Russian)
Trudy Sem. Petrovsk. No. 12 (1987), 37-46, 242; translation in J. Soviet Math. 47 (1989), no. 3,
2509–2516.

\bibitem{A3} V.I. Arnold, \textit{Critical point of smooth functions}. Vancouver Intern. Congr. of Math., 1974, vol.1, 19-39.

\bibitem{AKh}
V.I. Arnold and B.A. Khesin,
\textit{Topological Methods in Hydrodynamics}. Springer-Verlag,
New York 1998.

\bibitem{CM} R. Cardona and E. Miranda, \textit{Integrable systems and closed one forms}. J. Geom. Phys. 131 (2018), 204-209.

\bibitem{CM2} R. Cardona and E. Miranda, \textit{On the volume elements of a manifold with transverse zeroes}. Regular and Chaotic Dynamics, 2019, vol. 24, no. 2, pp. 187-197.

\bibitem{CV} K. Cieliebak and E. Volkov, \emph{A note on the stationary Euler equations of hydrodynamics}. Ergodic Theory and Dynamical Systems, 37(2), 454-480. 

\bibitem{EG}J. Etnyre and R. Ghrist. \emph{Contact topology and hydrodynamics III, Knotted orbits}. Trans. Amer. Math. Soc. 352 (2000) 5781-5794.

\bibitem{EG2} { J. Etnyre and R. Ghrist, \emph{Stratified integrals and unknots in inviscid flows}. Cont. Math. 246, 99-112 (1999).}

\bibitem{G2} R. Ghrist, \emph{On the Contact Topology and Geometry of Ideal Fluids}. Handbook of Mathematical Fluid Dynamics IV, North Holland 2007.

\bibitem{GK} V.L. Ginzburg and B. Khesin, \emph{Steady fluid flows and symplectic geometry}. Journal of Geometry and Physics, 14(2):195-210, 1994.

\bibitem{GL}
		M. Gualtieri and S. Li,
		\emph{Symplectic groupoids of log symplectic manifolds}. International Mathematics Research Notices, 2014(11), 3022-3074.

\bibitem{GMP} V. Guillemin, E. Miranda and A.R. Pires, \emph{Symplectic and Poisson geometry on $b$-manifolds}. Advances in mathematics 264 (2014): 864-896.

    \bibitem{Deblog} V. Guillemin, E. Miranda and J. Weitsman, \emph{Desingularizing $b^m$-symplectic manifolds}. Int. Math. Res. Not, 2017, rnx126, 13pp.

 \bibitem{H1} M. Hutchings, \emph{Reidemeister torsion in generalized Morse theory.} Thesis (Ph.D.) Harvard University, 1998, 78 pp, ISBN: 978-0591-85470-1.

  \bibitem{H2} M. Hutchings, \emph{Reidemeister torsion in generalized Morse theory}. Forum Math. 14 (2002), no. 2, 209--244.



\bibitem{McS} D. McDuff and D. Salamon, \emph{Introduction to Symplectic Topology}. Oxford University Press, New
York, 1995.

\bibitem{Me} R. Melrose, {\em The Atiyah Patodi Singer Index Theorem}. Res. Not. Math., A. K. Peters, {1993}.

\bibitem{MO} E. Miranda and C. Oms, \emph{Contact structures with singularities}. arXiv preprint arXiv:1806.05638 (2018) (updated version available here \url{https://mat-web.upc.edu/people/eva.miranda/research.html})


         \bibitem{MP}    E. Miranda and A. Planas, \emph{Equivariant Classification of $b^m$-symplectic Surfaces}. Regul. Chaotic Dyn. 23 (2018), no. 4, 355-371.


\bibitem{P1} D. Peralta-Salas, \emph{Selected topics on the topology of ideal fluid flows}. Int. J. Geom. Methods Mod. Phys. 13 (2016) 1630012.

\bibitem{S} G. Scott, \textit{The Geometry of $b^k$ Manifolds}. J. Symplectic Geom. 14 (2016), no. 1, 7195.

\bibitem{Sm} S. Smale, \emph{Morse inequalities for a dynamical system}. Bull. Amer. Math. Soc. 66 1960 43-49.

\bibitem{T} D. Tischler, \textit{On fibering certaing foliated manifolds over $S^1$}. Topology 9 153-154, 1970.


\bibitem{weinstein} A. Weinstein, \emph{On the hypotheses of Rabinowitz' periodic orbit theorems}. J. Differential Equations 33 (1979), no. 3, 353-358.

\end{thebibliography}
\end{document}